\documentclass[12pt,reqno]{amsart}
\usepackage{graphicx}
\usepackage{epstopdf}
\usepackage{url}
\usepackage{latexsym}
\usepackage{amsfonts,amsmath,amssymb}
\newtheorem{theorem}{Theorem}[section]
\newtheorem{thm}{Theorem}[section]

\newtheorem{prop}[theorem]{Proposition}
\newtheorem{lemma}[theorem]{Lemma}

\newtheorem{remark}{Remark}

\newcommand{\Single}{\rm{Single}}

\newcommand{\stirl}[2]{\genfrac{\{}{\}}{0pt}{}{#1}{#2}}

\def\C{\mathbb{C}}

\def\N{\mathbb{N}}

\numberwithin{equation}{section}
\date{\today}
\usepackage{color}
\usepackage[dvipsnames]{xcolor}

\usepackage{tikz}

\def\and{\mathrm{d}}
\makeatletter
\newsavebox{\@brx}
\newcommand{\llangle}[1][]{\savebox{\@brx}{\(\m@th{#1\langle}\)}%
  \mathopen{\copy\@brx\kern-0.5\wd\@brx\usebox{\@brx}}}
\newcommand{\rrangle}[1][]{\savebox{\@brx}{\(\m@th{#1\rangle}\)}%
  \mathclose{\copy\@brx\kern-0.5\wd\@brx\usebox{\@brx}}}
\makeatother
\usepackage[pagebackref]{hyperref}
\allowdisplaybreaks[4]

\begin{document}

\title[An explicit formula
for Ramanujan's polynomials]{
Proof of an explicit formula for a series from  Ramanujan's Notebooks via tree functions}

\author{Ming-Jian Ding}
\address[Ming-Jian Ding]{School of Mathematic Sciences,
 Dalian University of Technology, Dalian 116024, P. R. China}
\email{ding-mj@hotmail.com}

\author[Jiang Zeng]{Jiang Zeng}
\address[Jiang Zeng]{Univ Lyon, Universit\'e Claude Bernard Lyon 1, CNRS UMR 5208,
 Institut Camille Jordan, 43 blvd. du 11 novembre 1918, F-69622 Villeurbanne cedex, France}
\email{zeng@math.univ-lyon1.fr}
\date{\today}
\begin{abstract}
We prove a recent conjecture, due to Vigren and Dieckmann, about an explicit triple sum formula for a series from Ramanujan's Notebooks.
We shall give two proofs: the first one is by evaluation and based on the
identity
\begin{equation*}
   \sum_{k=0}^\infty \frac{(x+k)^{m+k}}{k!}e^{-u(x+k)} u^k
 = \sum_{j=0}^\infty \sum_{i=0}^{m}\binom{m+j}{i}
 \stirl{m+j-i}{j}x^iu^j,
\end{equation*}
where $\genfrac\{\}{0pt}{}{n}{k}$ is a Stirling number of the second kind, and the second one is combinatorial in nature and by induction.
\end{abstract}

\maketitle

\section{Introduction}

A \emph{labelled tree} of size $n$ is a rooted tree comprising $n$ nodes that are labelled by distinct integers of the set $[n]:=\{1, \ldots, n\}$.
A classical result due to A. Cayley in 1899 states that the number of labelled non-plane trees with $n$ nodes is $n^{n-1}$.
This result can be derived using the combinatorial theory of formal power series as in the following.
If $T(z)$ is  the exponential generating function of such trees,
then it is the power series in $\C[[z]]$ satisfying $T(z)=ze^{T(z)}$, by \emph{inversion}~\cite{Ge16},
the so-called \emph{tree-function} $T(z)$ is given by
\begin{equation}\label{taylor+T+z}
T(z)=\sum_{n= 1}^{\infty}  n^{n-1} \frac{z^n}{n!}.
\end{equation}
Let $u=T(z)$, then $z=ue^{-u}$.
The following are two allied identities~\cite{Ze99, Ge16}:
\begin{align}
\frac{e^{xu}}{1-u}&=\sum_{n=0}^{\infty} (x+n)^n \frac{z^n}{n!},     \label{u1}\\
\frac{e^{xu}-1}{x}&=\sum_{n= 1}^{\infty} (x+n)^{n-1}\frac{z^n}{n!}. \label{u2}
\end{align}
For an arbitrary integer $m$, it is known that the series $\sum_{n\geq 0}
(x+n)^{n+m} \frac{z^n}{n!}$ can be expressed in terms of $u$.
Actually, in Entry 16 of Chapter 3 of his second notebook \cite{BEW83}
Ramanujan studied the numbers $\psi_k(m,x)$ defined by
\begin{equation}\label{ramanujan:formula}
   \sum_{k=0}^{\infty}\frac{(x+k)^{m+k}}{k!}e^{-u(x+k)}u^k
 = \sum_{k=1}^{m+1}\frac{\psi_k(m,x)}{(1-u)^{m+k}}.
\end{equation}
Set $Q_{m,k}(x)=\psi_{k+1}(m-1,x+m)$, which satisfies the recurrence relation~\cite{Ze99}
\begin{align}\label{Q-psi:ramanujan}
Q_{m,k}(x)=(x-k+1)Q_{m-1,k}(x+1)+(m+k-2)Q_{m-1,k-1}(x+1),
\end{align}
where $Q_{1,0}(x)=1$ and $k\in \{0, \ldots, m-1\}$.
Shor~\cite{Sh95} and Dumont and Ramamonjisoa~\cite{DR96}
found combinatorial interpretations for $Q_{m,k}(x)$ when $x=1$ and when $x\in \{-1,0,1\}$, respectively,
Zeng~\cite{Ze99} considered the case of general $x$ and
provided the combinatorial setting in certain sets of rooted labelled trees,
see \cite{CG01,GZ07, LZ14,  ES20, CY21, So23} for further related works.

In a survey of the one-variable Lagrange inversion formula,
Gessel~\cite[Theorem 3.2.5]{Ge16} proved that
there exists a polynomial $R_m(u,x)$, with integer coefficients,
of degree $m$ in $u$ and  $x$, respectively, such that
\begin{equation}\label{gessel:formula:positive}
   \sum_{k=0}^\infty \frac{(x+k)^{m+k}}{k!}e^{-u(x+k)}u^k
 = \frac{R_m(u,x)}{(1-u)^{2m+1}}.
\end{equation}
An alternative proof of Gessel's result goes
as follows.
From \eqref{ramanujan:formula} and \eqref{gessel:formula:positive} we derive
\begin{equation}
  R_{m}(u,x)=\sum_{k=0}^mQ_{m+1,k}(x-m-1)(1-u)^{m-k}.
\end{equation}
By \eqref{Q-psi:ramanujan}
the polynomial $Q_{m,k}(x)$ is of degree $m-k-1$ in $x$. Thus, the highest degree
of $x$ and $u$ must appear in $Q_{m+1,0}(x-m-1)(1-u)^{m}$.
The first three polynomials $R_m(u,x)$ are
\begin{align*}
R_0(u,x)&=1,\\
R_1(u,x)&=x+(1-x)u,\\
R_2(u,x)&=x^2+(1+3x-2x^2)u+(2-3x+x^2)u^2.
\end{align*}
However, no explicit finite formula for the polynomial $R_m(u,x)$ seems to be known for general $m$.
The aim of this note is equivalent to prove an explicit finite triple sum formula for $R_{m}(u,x)$, see Theorem~\ref{conj-VD}, 
which was recently conjectured by Vigren and Dieckmann~\cite{VD22}.
% see~\eqref{VD:formula} and~\eqref{iden+r+m}.

Let us first introduce some definitions.
The 2-\emph{associated Stirling subset number} $\stirl{n}{k}_{\geq 2}$ is the number of partitions of an
$n$-elements set into $k$ blocks,
each of which has at least two elements, see \cite{ES20} and \cite[A008299]{Slo},
it follows that $\stirl{n}{k}_{\geq 2}=0$ if $2k>n$.
By convention $\stirl{0}{k}_{\geq 2}=\delta_{0k}$.
It is easy to see that
\begin{equation*}
\stirl{n}{k}_{\geq 2}=k\stirl{n}{k}_{\geq 2}+(n-1)\stirl{n-2}{k-1}_{\geq 2} \quad \text{for} \quad n\geq 1.
\end{equation*}
The following explicit formula is also known~~\cite[A008299]{Slo}:
\begin{equation}\label{f-nk}
\stirl{n}{k}_{\geq 2}=\sum_{j=0}^k (-1)^j \binom{n}{j}\stirl{n-j}{k-j},
\end{equation}
where the curly brackets are used to denote the Stirling numbers of the second kind.

\begin{thm}[Conjecture of Vigren and Dieckmann~\cite{VD22}]\label{conj-VD}
\begin{align}\label{VD:formula}
   \sum_{k=0}^\infty \frac{(x+k)^{m+k}}{k!}e^{-u(x+k)} u^k
 = \sum_{k=0}^m\sum_{i=0}^{m-k}\binom{m+k}{i}\stirl{m+k-i}{k}_{\geq 2}\,\frac{x^i u^{k}}{(1-u)^{m+k+1}}.
\end{align}
Invoking \eqref{gessel:formula:positive}, this is equivalent to
\begin{equation}\label{iden+r+m}
   R_m(u,x)
 = \sum_{k=0}^m\sum_{i=0}^{m-k}\binom{m+k}{i}
   \stirl{m+k-i}{k}_{\geq 2}\,x^i u^{k} (1-u)^{m-k}.
\end{equation}
\end{thm}

When $u=0$ and $u=1$, the above formula reduces to
\begin{equation*}
R_m(0,x)=x^m \quad\text{and}\quad
R_m(1,x)=(2m-1)(2m-3)\cdots 3\cdot 1.
\end{equation*}
Recall that the \emph{second-order Eulerian numbers} $\llangle[\big]{m\atop k}\rrangle[\big]$ \cite[A008517]{Slo} are defined by
\begin{equation*}
   \llangle[\bigg]{m\atop k}\rrangle[\bigg]
 = (k+1)\llangle[\bigg]{m-1\atop k}\rrangle[\bigg]+(2m-k-1)\llangle[\bigg]{m-1\atop k-1}\rrangle[\bigg]
\end{equation*}
with $\llangle{0\atop 0}\rrangle=1$ and $0\leq k\leq m$.
When $x=0$, Equation~\eqref{iden+r+m} reduces to
\begin{equation}\label{R0}
  R_m(u,0)=\sum_{k=0}^m
  \stirl{m+k}{k}_{\geq 2}\,u^{k}(1-u)^{m-k}=\sum_{k=0}^{m}
  \llangle[\bigg]{m\atop k}\rrangle[\bigg] u^{k+1},
\end{equation}
where the last equality is equivalent to \cite[Corollary 3]{Smi00} by substituting $u=\frac{\lambda}{1+\lambda}$.
When $x=1$, Carlitz~\cite{Ca65,Smi00} proved that
\begin{equation}\label{R1}
  R_m(u,1)=\sum_{k=0}^{m}
  \llangle[\bigg]{m\atop k}\rrangle[\bigg] u^{k}.
\end{equation}
From \eqref{R0} and \eqref{R1} we obtain
\begin{align}
R_m(u,0)=u\cdot R_m(u,1)\qquad (m\geq 1).
\end{align}

\begin{remark}\rm
A geometric interpretation for  the second equality in \eqref{R0} can be given as in the following.
The number $\stirl{m+k}{k}_{\geq 2}$ counts the $k$-dimensional faces in the tropical Grassmannian of lines $G(2,m+1)$
and the corresponding $h$-vector is the second-order Eulerian numbers $\llangle[\big]{m\atop k}\rrangle[\big]$,
see \cite[Section 4]{SS04} or \cite[A134991]{Slo}.
\end{remark}

We shall present two proofs of Theorem~\ref{conj-VD} in the next two sections.
The first proof uses the explicit formula of the Stirling numbers of the second kind,
and the second proof works by means of a recurrence relation of generating function from a combinatorial model.

\section{First proof of Theorem~\ref{conj-VD}}
\begin{lemma}\label{lem+Rama+expre}
We have
\begin{equation*}
   \sum_{k=0}^\infty \frac{(x+k)^{m+k}}{k!}e^{-u(x+k)} u^k
 = \sum_{j=0}^\infty \sum_{i=0}^{m}\binom{m+j}{i}\begin{Bmatrix}m+j-i\\j\end{Bmatrix}x^iu^j.
\end{equation*}
\end{lemma}
\begin{proof}

Recall the well-known formula of the Stirling numbers of the second kind
\begin{equation}\label{explicit-stirling-formula}
\genfrac\{\}{0pt}{}{n}{k}
=\frac{1}{k!}\sum_{i=0}^k(-1)^{k-i}\binom{k}{i}i^n.
\end{equation}
Substituting $e^{-u(x+k)}=\sum_{\ell=0}^\infty (-1)^{\ell} (x+k)^{\ell}u^{\ell}/{\ell}!$ and replacing $k+\ell$ by $j$ we have
\begin{align*}
 \sum_{k=0}^\infty \frac{(x+k)^{m+k}}{k!}e^{-u(x+k)} u^k
 &= \sum_{j=0}^\infty \frac{u^j}{j!}\sum_{{\ell}=0}^{j}(-1)^{j-{\ell}}\binom{j}{{\ell}}({\ell}+x)^{m+j}  \\
 &= \sum_{j=0}^\infty \sum_{i=0}^{m}\binom{m+j}{i}\frac{u^j}{j!}\sum_{{\ell}=0}^{j}(-1)^{j-{\ell}}\binom{j}{{\ell}}{\ell}^{m+j-i}x^i \\
 &= \sum_{j=0}^\infty \sum_{i=0}^{m}\binom{m+j}{i}\begin{Bmatrix}m+j-i\\j\end{Bmatrix}x^iu^j,
\end{align*}
where the last equality follows from \eqref{explicit-stirling-formula}.
\end{proof}

For $m\in \N$ introduce the short-hand notation
\begin{align}\label{def+g+Rama}
g_m(u,x):=\sum_{k=0}^m\sum_{i=0}^{m-k}
 \binom{m+k}{i}\stirl{m+k-i}{ k}_{\geq 2}\,\frac{x^i u^{k}}
{(1-u)^{m+k+1}}.
\end{align}

\begin{lemma}\label{lem+g+expre}
We have
\begin{equation*}
   g_m(u,x)
 = \sum_{j=0}^\infty \sum_{i=0}^{m}\binom{m+j}{i}\genfrac\{\}{0pt}{}{m+j-i}{j}x^iu^j.
\end{equation*}
\end{lemma}
\begin{proof}
Writing $1=(1-u)^{m+j+1}\cdot (1-u)^{-(m+j+1)}$ and
\begin{equation*}
     (1-u)^{m+j+1}
   = \left(1+\frac{u}{1-u}\right)^{-(m+j+1)}
 \\= \sum_{\ell=0}^{\infty}(-1)^{\ell} {m+j+l\choose \ell}\left(\frac{u}{1-u}\right)^{\ell},
\end{equation*}
we have
\begin{align*}
     &\quad \sum_{j=0}^{\infty}\sum_{i=0}^{m}\binom{m+j}{i}\begin{Bmatrix}m+j-i\\j\end{Bmatrix}x^i u^{j}\\
%    &= \sum_{i=0}^{\infty}\sum_{k=0}^{m}\sum_{j=0}^{k}
%       (-1)^{k-j}\binom{m+k}{k-j}
%       \left(\frac{u}{1-u}\right)^{k-j}\binom{m+j}{i}
%       \genfrac\{\}{0pt}{}{m+j-i}{j}\frac{x^iu^{j}}{(1-u)^{m+j+1}} \\
     &= \sum_{i=0}^{m}\sum_{j=0}^{\infty}\sum_{\ell=0}^{\infty}(-1)^{\ell}\binom{m+j+\ell}{\ell}
       \left(\frac{u}{1-u}\right)^{\ell}\binom{m+j}{i}
       \genfrac\{\}{0pt}{}{m+j-i}{j}\frac{x^iu^{j}}{(1-u)^{m+j+1}}\\
     &= \sum_{i=0}^{m}\sum_{k=0}^{\infty}\sum_{j=0}^{k}
       (-1)^{k-j} \binom{m+k}{i}\binom{m+k-i}{k-j}\begin{Bmatrix}m+j-i\\j\end{Bmatrix}\frac{x^iu^{k}}{(1-u)^{m+k+1}}.
%    &= \sum_{j=0}^{\infty}\sum_{i=0}^{\infty}
%       \left(1+\frac{u}{1-u}\right)^{-(m+j+1)}\binom{m+j}{i}\genfrac\{\}{0pt}{}{m+j-i}{j}\frac{x^iu^{j}}{(1-u)^{m+j+1}}
\end{align*}
where the last equality follows from the substitution $\ell=k-j$ and the identity
\begin{equation*}
   \binom{m+k}{i}\binom{m+k-i}{k-j} = \binom{m+j}{i}\binom{m+k}{k-j}.
\end{equation*}
Substituting $j$ by $k-j$ in the last summation and invoking \eqref{f-nk} we obtain
\begin{equation}\label{sum-k}
%   \quad \sum_{j=0}^{\infty}\sum_{i=0}^{m}\binom{m+j}{i}\begin{Bmatrix}m+j-i\\j\end{Bmatrix}x^i u^{j}
  \sum_{k=0}^\infty\sum_{i=0}^{m}\binom{m+k}{i}
  \stirl{m+k-i}{k}_{\geq 2}\frac{x^iu^{k}}{(1-u)^{m+k+1}}.
\end{equation}
Note that  $\stirl{m+k-i}{k}_{\geq 2}$ is zero if $k>m$ or $i > m-k$.
So, the ranges of $k$ and $i$ in \eqref{sum-k} are bounded by $m$ and $m-k$, respectively.
%\begin{equation*}
% g_m(u,x)=\sum_{k=0}^m\sum_{i=0}^{k}\binom{2m-k}{i}
% f(2m-k-i,m-k)\frac{x^iu^{m-k}}{(1-u)^{2m+1-k}}.
%\end{equation*}
%This is the desired formula.
\end{proof}

Combining Lemma~\ref{lem+Rama+expre} and~\ref{lem+g+expre} completes the proof of Theorem~\ref{conj-VD}.

\section{Second proof of Theorem~\ref{conj-VD}}
Let
$$g(m,u,x):= \sum_{k=0}^\infty \frac{(x+k)^{m+k}}{k!}e^{-u(x+k)} u^k.$$
It is easy to verify that  
\begin{equation}\label{lem+rec+g}
g(m, u, x+1)-g(m, u,x)=\frac{\partial}{\partial u} g(m-1, u,x).
\end{equation}
Let 
\begin{equation}\label{pmki}
  p(m,k,i) = \binom{m+k}{i}\stirl{m+k-i}{k}_{\geq 2}.
\end{equation}
Then the polynomial
\begin{equation}\label{poly-pm}
  p_m(u,x) := \sum_{k=0}^{m}\sum_{i=0}^{m-k}p(m,k,i) u^kx^i
\end{equation}
is related to $g_m(u,x)$ (cf.
\eqref{def+g+Rama}) by 
\begin{equation}\label{rel+g+p}
  g_m(u,x) = \frac{1}{(1-u)^{m+1}}p_m\left(\frac{u}{1-u},x\right).
\end{equation}

Thus, to prove Theorem~\ref{conj-VD}, i.e., 
$g(m,u,x)=g_m(u,x)$, it is sufficient to show that  
$g_m(u,x)$ satisfies \eqref{lem+rec+g}, which means, in terms of $p_m(u,x)$,
\begin{equation}\label{pde+p+m}
  p_m(u, x+1)-p_m(u,x)=\biggl[m+(1+u)\frac{\partial}{\partial{u}}\biggr]p_{m-1}(u,x).
\end{equation}
By \eqref{poly-pm},  Equation~\eqref{pde+p+m} is equivalent to
\begin{gather*}
           \sum_{k=0}^m\sum_{i=0}^{m-k}p(m,k,i)u^k((x+1)^i-x^i)\\
         = \sum_{k=0}^{m-1}\sum_{i=0}^{m-k-1}\left(mp(m-1,k,i)   +(k+1)p(m-1,k+1,i)+kp(m-1,k,i)\right)u^k x^i.
\end{gather*}
Extracting  the coefficients of $u^{k}x^{i}$ on both sides, we get
\begin{equation}\label{iden+p+inv}
     \sum_{j=i+1}^{m-k}\binom{j}{i}p(m,k,j)=(k+1)p(m-1,k+1,i)+(m+k)p(m-1,k,i).
\end{equation}

It remains to prove \eqref{iden+p+inv}.
To do this,
we show that both sides of~\eqref{iden+p+inv} are the cardinality of certain set  by double counting.
Let $\Pi(m,k,i)$ denote the set of partitions of $\{1, \ldots, m+k\}$
into $i$ singletons  and $k$ blocks with size at least 2.
It is clear from \eqref{pmki} that  $p(m,k,i)$ is the cardinality of  $\Pi(m,k,i)$. Let $\Pi_0(m,k+1,i)$ denote the set of partitions  of $\{0, 1, \ldots, m+k\}$ into  $i$ singletons and 
$k+1$ blocks of size at least 2, of which one  contains 0, called \emph{0-block}.

\begin{prop}\label{prop+stirling+p}
For nonnegative integers $n$, $k$ and $i$, the two sides of \eqref{iden+p+inv} are the cardinality of  $\Pi_0(m,k+1,i)$.
\end{prop}
\begin{proof}
We construct the partitions in 
$\Pi_0(m,k+1,i)$ as follows:
\begin{itemize}
\item  the 0-block  is a doubleton; clearly we can choose any element in $\{1, \ldots, m+k\}$ to form the douleton with 0,  and
using  the remaining elements to make 
 a partition in $\Pi(m-1,k,i)$, there are $(m+k)p(m-1,k,i)$ ways;
\item  the 0-block  has at least three elements;
from  any partition  in  $\Pi(m-1,k+1,i)$, pick up a non-singleton block, in $k+1$ ways, to make it a 0-block, and the remaing blocks form a partition in $\Pi(m-1,k+1,i)$. 
       Clearly the number of such partitions is $(k+1)p(m-1,k+1,i)$.
\end{itemize}

For the left-hand side of~\eqref{iden+p+inv}, we set up a bijection
$\phi: \Omega\to \Pi_0(m,k+1, i)$, where $\Omega:=\cup_{j>i}\Omega_j$ with
\begin{equation*}
\Omega_j=\{(\pi, T): \pi\in \Pi(m,k,j), \; T \subseteq {\Single(\pi)} \;\text{with}\,\;
|T|=j-i\},
\end{equation*}
where $\Single(\pi)$ is the set of elements in singeltons of $\pi$, as in the following: join 0 to $T$ to make a 0-block $T_0$.
Together with the remaining blocks of $\pi$ we get a partition
$\pi_0$ in $\Pi_0(m,k+1, i)$.
It is clear that the cardinality of $\Omega$ is the left-hand side of \eqref{iden+p+inv}.
\end{proof}

\section*{Acknowledgements}
The first author was supported by the \emph{China Scholarship Council}.
This work was done during his visit  at  Universit\'e Claude Bernard Lyon 1 in 2022-2023.

\end{document}